\documentclass[12pt, reqno]{amsart}
\usepackage[margin=1in]{geometry}
\usepackage{amssymb,latexsym,amsmath,amscd,amsfonts}
\usepackage{latexsym}
\usepackage[mathscr]{eucal}
\usepackage{bm}
\usepackage{mathptmx}
\usepackage{amssymb}
\usepackage{amsthm}
\usepackage{dcolumn}
\usepackage[all]{xy}
\usepackage{enumitem}
\usepackage[utf8]{inputenc}

\def\textmatrix#1&#2\\#3&#4\\{\bigl({#1 \atop #3}\ {#2 \atop #4}\bigr)}
\def\dispmatrix#1&#2\\#3&#4\\{\left({#1 \atop #3}\ {#2 \atop #4}\right)}
\newcommand{\beg}{\begin{equation}}
	\newcommand{\eeg}{\end{equation}}
\newcommand{\ben}{\begin{eqnarray*}}
	\newcommand{\een}{\end{eqnarray*}}

\newtheorem{thm}{Theorem}[section]
\newtheorem{cor}[thm]{Corollary}
\newtheorem{lem}[thm]{Lemma}

\newtheorem{prop}[thm]{Proposition}
\numberwithin{equation}{section} \theoremstyle{definition}
\newtheorem{defn}[thm]{Definition}

\newcommand{\HS}{\mathcal H}
\newcommand{\KS}{\mathcal K}

\newcommand{\T}{\mathbb{T}}

\newcommand{\Z}{\mathbb{Z}}
\newcommand{\N}{\mathbb{N}}

\newcommand{\ov}{\overline}
\newcommand{\la}{\left\langle}
\newcommand{\ra}{\right\rangle}

\begin{document}
	\title[$*$-regular dilations for $q$-commuting contractions]{The $q$-commuting contractions and $*$-regular dilations}

	\author[Tomar] {NITIN TOMAR}
	
	\address[Nitin Tomar]{Department of Mathematics, Indian Institute of Technology Bombay, Powai, Mumbai-400076, India.} \email{tomarnitin414@gmail.com}		
	
	
	
	\thanks{The author is supported through the `Core Research Grant (CRG)', Award No. CRG/2023/005223, granted to Professor Sourav Pal by the Science and Engineering Research Board (SERB)}	
	
	\begin{abstract}
		A tuple $\underline{T}=(T_1, \dotsc, T_k)$ of contractions on a Hilbert space $\mathcal{H}$ is said to be \textit{$q$-commuting with} $\|q\|=1$ if there is a family of scalars $q=\{q_{ij} \in \mathbb{C} : |q_{ij}|=1, \ q_{ij}=q_{ji}^{-1}, \  1 \leq i < j \leq k \}$ such that $T_i T_j =q_{ij}T_j T_i$  for $1 \leq i < j \leq k$. In addition if $T_iT_j^*=\overline{q}_{ij}T_j^*T_i$ for $1 \leq i<j \leq k$, then $\underline{T}$ is said to be \textit{doubly $q$-commuting}. In this article, we introduce the notion of $*$-regular dilation of a $q$-commuting tuple of contractions $\underline{T}$ with $\|q\|=1$. The results presented in this article are motivated by the works of Ga\c spar–Suciu \cite{Gaspar} and Timotin \cite{Timotin} on commuting contractions. Indeed, we prove that a minimal $q$-isometric dilation $\underline{V}$ of a $q$-commuting tuple of contractions with $\|q\|=1$ is $*$-regular if and only if $\underline{V}$ is doubly $q$-commuting. Furthermore, we establish that a $q$-commuting tuple $\underline{T}$ of contractions with $\|q\|=1$ is doubly $q$-commuting if and only if $\underline{T}$ admits a minimal regular $q$-isometric dilation which is doubly $q$-commuting. We also obtain a structure theorem for $q$-commuting tuples of contractions with $\|q\|=1$ that admit a regular $q$-unitary dilation.
	\end{abstract}

	\maketitle

	\section{Introduction}
	
	\noindent 
	Throughout the paper, all operators are bounded linear maps acting on complex Hilbert spaces. We denote by $\mathbb{C}, \mathbb{D}$ and $\mathbb{T}$ the complex plane, the unit disk and the unit circle in the complex plane, respectively, with center at the origin. $\mathcal{B}(\HS)$ denotes the algebra of operators acting on a Hilbert space $\HS$. The symbols $\N$, $\Z$ and $\Z_+$ denote the sets of natural numbers, integers and non-negative integers, respectively.  A \textit{contraction} is an operator with norm at most $1$. B. Sz.-Nagy  \cite{NagyII} proved that for every contraction $T$ acting on a Hilbert space $\HS$, there exist a Hilbert space $\mathcal{K}$ containing $\mathcal{H}$ and a unitary $U$ on $\mathcal{K}$ such that
	\[
	T^k=P_\HS U^k|_\HS \quad (k=0, 1, 2, \dotsc), 
	\]
	where $P_\HS$ denotes the orthogonal projection of $\mathcal K$ onto $\HS$. In operator theoretic language, we call it a \textit{unitary dilation} of a contraction. Ando \cite{Ando} proved the extension of Sz.-Nagy's dilation theorem for a single contraction to any commuting pair of contractions. However, Parrott \cite{Par} provided a counter-example to show that such a dilation theorem does not hold for more than two commuting contractions. 
	Since then, many efforts have been devoted to characterize the families of $n$-tuples of commuting contractions that admit unitary dilations (see \cite{Agler, Bal:Tim:Tre, Bal:Tre:Vin, Cur:Vas 1, Cur:Vas 2} and the references therein). In this context, Brehmer \cite{Brehmer} introduced a stronger notion of unitary dilation, known as regular unitary dilation, and provided a characterization of commuting families of contractions that admit regular unitary dilations. This foundational work was further developed by Halperin \cite{Halp} and S. Nagy \cite{Nagy}, who offered further insights into the structure and properties of regular unitary dilations.

	\medskip

	While the classical dilation theory has primarily focused on commuting families of contractions, subsequent work has extended these ideas to a more general non-commuting setting. In particular, many results have been established for 
	$q$-commuting contractions with $\|q\|=1$. In this direction, an interested reader is referred to \cite{Ball, Barik, Bis:Pal:Sah, Dey, DeyI, Jeu, K.M., PalII, PalIII, Sebestyen} and references therein. For decomposition results in this context, we refer to \cite{Maji, RI, RII, Pal}.
	
	\begin{defn}
		A family $\mathcal{T}=\{T_\alpha: \alpha \in \Lambda \}$ of operators acting on a Hilbert space $\mathcal{H}$ is said be \textit{$q$-commuting}  with $\|q\|=1$ if there exists a family of scalars $q=\{q_{\alpha \beta} \in \T \ :  \ q_{\alpha \beta}=q_{\beta \alpha}^{-1}, \  \alpha \ne \beta  \ \text{in} \ \Lambda \}$ such that $T_\alpha T_\beta=q_{\alpha \beta}T_\beta T_\alpha$ for all $\alpha, \beta$ in $\Lambda$ with $\alpha \ne \beta$. In addition, if $T_\alpha T_\beta^*=\overline{q}_{\alpha \beta} T_\beta^* T_\alpha$, then $\mathcal{T}$ is said to be \textit{doubly $q$-commuting}. 
	\end{defn}
	
	In this article, we are mainly concerned with $q$-commuting tuples of contractions with $\|q\|=1$. Accordingly, we now recall the notion of $q$-isometric and $q$-unitary dilation for such tuples.
	
	\begin{defn} Let $\underline{T}=(T_1, \dotsc, T_k)$ be a $q$-commuting tuple of contractions with $\|q\|=1$ acting on a Hilbert space $\HS$. A $q$-commuting tuple of isometries $\underline{V}=(V_1, \dotsc, V_k)$ with $\|q\|=1$, acting on a Hilbert space $\mathcal{K} \supseteq \HS$, is said to be a \textit{$q$-commuting isometric dilation}, or simply a \textit{$q$-isometric dilation}, of $\underline{T}$ if
		\[
		T_1^{m_1}\dotsc T_k^{m_k}=P_\HS V_1^{m_1}\dotsc V_k^{m_k}|_{\HS} \quad \text{for all  $m_1, \dotsc, m_k \in \Z_+$.}
		\]	
		The dilation is called \textit{minimal} if 
		\[
		\mathcal{K}=\ov{\text{span}}\left\{V_1^{m_1}\dotsc V_k^{m_k}h : m_1, \dotsc, m_k \in \Z_+, h \in \HS \right\}.
		\]
		A $q$-commuting tuple of unitaries $\underline{U}=(U_1, \dotsc, U_k)$ with $\|q\|=1$, acting on a Hilbert space $\mathcal{K}$ containing $\HS$, is said to be a \textit{$q$-commuting unitary dilation}, or simply a \textit{$q$-unitary dilation}, of $\underline{T}$ if
		\[
		T_1^{m_1}\dotsc T_k^{m_k}=P_\HS U_1^{m_1}\dotsc U_k^{m_k}|_{\HS} \quad \text{for all  $m_1, \dotsc, m_k \in \Z_+$.}
		\]	
		In this case, the dilation is called \textit{minimal} if
		$
		\mathcal{K}=\ov{\text{span}}\left\{U_1^{m_1}\dotsc U_k^{m_k}h : m_1, \dotsc, m_k \in \Z, h \in \HS \right\}.
		$
	\end{defn}

	\medskip 
	
	Some of the earliest contributions to dilation theory in the $q$-commuting setting are due to Sebesty\'{e}n \cite{Sebestyen}, and Keshari and Mallick \cite{K.M.}. The author of  \cite{Sebestyen} proved that an anticommuting pair of contractions on a Hilbert space admits a dilation to an anticommuting pair of unitaries. Furthermore, Keshari and Mallick \cite{K.M.} generalized this result to any $q$-commuting pair of contractions with $\|q\|=1$.  In the multivariable setting, the authors of \cite{Barik} (see Theorem 3.10 in \cite{Barik}) proved that any $q$-commuting tuple $\underline{T}=(T_1, \dotsc, T_k)$ of contractions with $\|q\|=1$ admits a dilation to $q$-commuting $k$-tuple of isometries if $\underline{T}$ satisfies the following: 
	\begin{equation*}
		\underset{\{\alpha_1, \dotsc, \alpha_m\} \subset u}{\sum}(-1)^m(T_{\alpha_1}\dotsc T_{\alpha_m})(T_{\alpha_1}\dotsc T_{\alpha_m})^* \geq 0 \quad \text{for every $u \subseteq \{1, \dotsc, k\}$}.	
	\end{equation*}
	We also mention here that Theorem 6.2 in \cite{Jeu} (see also Corollary 6.2 in \cite{Ball}) guarantees a $q$-unitary extension of a doubly $q$-commuting tuple of isometries. Moreover, Theorem 6.1 in \cite{Ball} established the success of $q$-unitary extension for a $q$-commuting tuple of isometries with $\|q\|=1$, and Corollary 5.4 in \cite{PalII} extends it to an arbitrary $q$-commuting family of isometries with $\|q\|=1$.

	\medskip 
	
	The authors of \cite{PalII} introduced the notion of regular $q$-unitary dilation in the $q$-commuting setting. We say that a $q$-commuting family $\mathcal{T}=\{T_\alpha : \alpha \in \Lambda\}$ of contractions  with $\|q\|=1$ acting on a Hilbert space $\mathcal{H}$ admits a \textit{regular $q$-unitary dilation} if there exist a Hilbert space $\mathcal
	{K} \supseteq \mathcal{H}$ and a $q$-commuting family $\mathcal{U}=\{U_\alpha : \alpha \in \Lambda \}$ of unitaries acting on $\mathcal{K}$ such that
	\begin{align}\label{eqn_reg}
		\underset{1\leq i<j \leq k}\prod q_{\alpha_{i}\alpha_{j}}^{-m_{\alpha_{i}}^+m_{\alpha_{j}}^-}\left[(T_{\alpha_{1}}^{m_{\alpha_{1}}^-})^*\dotsc (T_{\alpha_{k}}^{m_{\alpha_{k}}^-})^*\right]\left[T_{\alpha_{1}}^{m_{\alpha_{1}}^+}\dotsc T_{\alpha_{k}}^{m_{\alpha_{k}}^+}\right]=P_\mathcal{H}U_{\alpha_{1}}^{m_{\alpha_{1}}}\dotsc U_{\alpha_{k}}^{m_{\alpha_{k}}}|_\mathcal{H}
	\end{align}
	for every $m_{\alpha_1}, \dotsc, m_{\alpha_k} \in \mathbb{Z}$ and $\alpha_1,  \dotsc, \alpha_k$ in $\Lambda$. Here, $m^+=\max \{m, 0\}$ and $m^-=-\min\{m, 0\}$ for $m \in \mathbb{Z}$. It is well-known that a commuting family of contractions satisfying Brehmer's positivity conditions admits a regular unitary dilation and vice-versa (see \cite{Brehmer, Halp}). The authors of \cite{PalII} generalized this result and proved that a $q$-commuting family $\mathcal{T}$ of contractions with $\|q\|=1$ admits a regular $q$-unitary dilation if and only if $\mathcal{T}$ satisfies Brehmer's positivity conditions.	Furthermore, the authors of \cite{PalII} found success of regular $q$-unitary dilation for several classes of $q$-commuting contractions with $\|q\|=1$ as mentioned below.
	
	\begin{thm}[\cite{PalII}, Theorem 1.13]\label{thm_001}
		Let $\mathcal{T}=\{T_\alpha : \alpha \in \Lambda\}$ be a $q$-commuting family of contractions with $\|q\|=1$ on a Hilbert space $\HS$. Then $\mathcal{T}$ admits a regular $q$-unitary dilation in each of the following cases:
		\begin{enumerate}
			\item $\mathcal{T}$ consists of $q$-commuting isometries;
			\item $\mathcal{T}$ consists of doubly $q$-commuting contractions;
			\item $\mathcal{T}$ is a countable family and ${\displaystyle \sum_{\alpha \in \Lambda} \|T_{\alpha}h\|^2 \leq \|h\|^2 }$ for all $h\in \HS$.
		\end{enumerate}
	\end{thm}
In this article, we primarily focus on $q$-commuting tuple of contractions with $\|q\|=1$. We generalize in Section \ref{sec_02} the following results of Ga\c spar-Suciu \cite{Gaspar} and Timotin \cite{Timotin} to  $q$-commuting tuples of contractions with $\|q\|=1$. 
	
	\begin{thm}[Ga\c spar-Suciu \cite{Gaspar} and Timotin \cite{Timotin}]\label{thm_104}
		Let $\underline{T}=(T_1, \dotsc, T_k)$ be a commuting tuple of contractions acting on a Hilbert space $\HS$ and let $\underline{V}=(V_1, \dotsc, V_k)$ be a minimal isometric dilation of $\underline{T}$. Then $\underline{V}$ is doubly commuting if and only if $\underline{V}$ is $*$-regular, i.e., 
		\begin{align}\label{eqn_102}
			\left[T_1^{m_1^+} \dotsc T_k^{m_k^+}\right]\left[T_1^{*m_1^-}\dotsc T_k^{*m_k^-}\right]= P_\mathcal{H}\left[V_1^{*m_1^-}\dotsc V_k^{*m_k^-}\right]\left[V_1^{m_1^+} \dotsc V_k^{m_k^+}\right]\bigg|_\mathcal{H}
		\end{align}
		for all $m_1, \dotsc, m_k \in \mathbb{Z}$. Here, $m^+=\max \{m, 0\}$ and $m^-=-\min\{m, 0\}$ for $m \in \mathbb{Z}$.  
	\end{thm}
 As an application of Theorem \ref{thm_104}, we have the following proposition due to Ga\c spar-Suciu \cite{Gaspar}. 
	
	\begin{prop}[\cite{Gaspar}, Proposition 0]\label{prop_105}
		Let $\underline{T}=(T_1, \dotsc, T_k)$ be a commuting tuple of contractions acting on a Hilbert space $\HS$. Then the following statements are equivalent:
		\begin{enumerate}
			\item $\underline{T}$ is doubly commuting;
			\item $\underline{T}$ admits a minimal regular isometric dilation which is doubly commuting;
			\item $\underline{T}$ admits a minimal regular unitary dilation $\underline{U} = (U_1, \ldots, U_k)$ such that $(U_1^*, \dotsc, U_k^*)$ is a minimal regular unitary dilation for $(T_1^*, \dotsc, T_k^*)$.
		\end{enumerate}
	\end{prop}
	
These seminal results due to Ga\c spar-Suciu \cite{Gaspar} and Timotin \cite{Timotin} explore the interplay between doubly commuting contractions and the notion of $*$-regularity. The desired generalization of these results in the $q$-commuting setting involve defining a suitable analog of $*$-regularity for a $q$-commuting tuple of contractions with $\|q\|=1$, followed by a series of long computations based on the works of Ga\c spar-Suciu \cite{Gaspar} and Timotin \cite{Timotin}. As an application, we obtain a structure theorem for $q$-commuting tuple of contractions with $\|q\|=1$ having a regular $q$-unitary dilation. 
	
	\section{Regular dilations and $*$-regular dilations in the $q$-commuting setting}\label{sec_02}
	
	\noindent An analog of regular unitary dilation for a $q$-commuting tuple of contractions with $\|q\|=1$ was introduced by the authors of \cite{PalII}. Let $\underline{T}=(T_1, \dotsc, T_k)$ be a $q$-commuting tuple of contractions with $\|q\|=1$ acting on a Hilbert space $\HS$. A $q$-unitary dilation $\underline{U}=(U_1, \dotsc, U_k)$ of $\underline{T}$ is said to be regular, if it satisfies the condition \eqref{eqn_reg} or equivalently,
	\[
	\left[T_1^{*m_1^-}\dotsc T_k^{*m_k^-}\right]\left[T_1^{m_1^+} \dotsc T_k^{m_k^+}\right]=P_\mathcal{H}\left[U_1^{*m_1^-}\dotsc U_k^{*m_k^-}\right]\left[U_1^{m_1^+} \dotsc U_k^{m_k^+}\right]\bigg|_\mathcal{H}
	\] 
	for all $m_1, \dotsc, m_k \in \mathbb{Z}$. As mentioned earlier, $m^+=\max \{m, 0\}$ and $m^-=-\min\{m, 0\}$ for $m \in \mathbb{Z}$. Building on the definition of regular $q$-unitary dilation, we now define regular $q$-isometric dilation.
	
	\begin{defn}
		Let $\underline{T}=(T_1, \dotsc, T_k)$ be a $q$-commuting tuple of contractions with $\|q\|=1$ acting on a Hilbert space $\HS$. A $q$-isometric dilation $\underline{V}=(V_1, \dotsc, V_k)$ is said to be \textit{regular}, if it satisfies the condition 
		\begin{align*}
			\left[T_1^{*m_1^-}\dotsc T_k^{*m_k^-}\right]\left[T_1^{m_1^+} \dotsc T_k^{m_k^+}\right]=P_\mathcal{H}\left[V_1^{*m_1^-}\dotsc V_k^{*m_k^-}\right]\left[V_1^{m_1^+} \dotsc V_k^{m_k^+}\right]\bigg|_\mathcal{H}
		\end{align*}
		for all $m_1, \dotsc, m_k \in \mathbb{Z}$.
	\end{defn}
	
	We recall from the literature \cite{Gaspar, Timotin} that an isometric (unitary) minimal dilation $\underline{V}=(V_1, \dotsc, V_k)$ of $\underline{T}=(T_1, \dotsc, T_k)$ is $*$-regular if  the condition \eqref{eqn_reg} holds. It is evident that $\underline{V}$ is a minimal regular unitary dilation of $\underline{T}$ if and only if $(V_1^*, \dotsc, V_k^*)$ is a minimal $*$-regular unitary dilation of $(T_1^*, \dotsc, T_k^*)$, i.e.,
	\begin{align}\label{eqn_*reg2}
		\ \left[T_1^{*m_1^+}\dotsc T_k^{*m_k^+}\right]\left[T_1^{m_1^-} \dotsc T_k^{m_k^-}\right]
		=P_\mathcal{H}\left[V_1^{m_1^-}\dotsc V_k^{m_k^-}\right]\left[V_1^{*m_1^+} \dotsc V_k^{*m_k^+}\right]\bigg|_\mathcal{H}
	\end{align}
	for all $m_1, \dotsc, m_k \in \mathbb{Z}$. The above observation plays a key role in defining a suitable analog of $*$-regularity for $q$-commuting tuple of contractions with $\|q\|=1$. To do so, we generalize \eqref{eqn_*reg2} in the $q$-commuting setting. We begin with the following lemma. 
	
	\begin{lem}\label{lem_107}
		Let $\underline{T}=(T_1, \dotsc, T_k)$ be a doubly $q$-commuting tuple of contractions. Then
		\[
		\left[T_1^{*m_1^-}\dotsc T_k^{*m_k^-}\right]\left[T_1^{m_1^+} \dotsc T_k^{m_k^+}\right]=\underset{1\leq i<j \leq k}{\prod}q_{ij}^{-m_i^-m_j^++m_i^+m_j^-} \left[T_1^{m_1^+} \dotsc T_k^{m_k^+}\right]\left[T_1^{*m_1^-}\dotsc T_k^{*m_k^-}\right].
		\]
	\end{lem}
	
	\begin{proof}
		For simplicity, we first prove the result for $k=2$. Let $(T_1, T_2)$ be a doubly $q$-commuting pair of contractions. Suppose $T_1T_2=q_{12}T_2T_1$ and $T_1T_2^*=\overline{q}_{12}T_2^*T_1$ for some $q_{12} \in \mathbb{T}$. For an operator $T$, it is clear that $T^{*m^-}T^{m^+}=T^{m^+}T^{*m^-}$ for all $m \in \Z$. By $q$-commutativity, we have that
		\begin{align*}
			T_1^{*m_1^-}\left(T_2^{*m_2^-}T_1^{m_1^+}\right)T_2^{m_2^+}
			=q_{12}^{m_1^+m_2^-}T_1^{m_1^+}T_1^{*m_1^-}T_2^{m_2^+}T_2^{*m_2^-} 	
			=q_{12}^{m_1^+m_2^-}q_{12}^{-m_1^-m_2^+}T_1^{m_1^+}T_2^{m_2^+}T_1^{*m_1^-}T_2^{*m_2^-}.
		\end{align*}		
		The general case now follows from the mathematical induction. 
	\end{proof}
	
	Let $\underline{U}=(U_1, \dotsc, U_k)$ be a minimal regular $q$-unitary dilation of a $q$-commuting tuple $\underline{T}=(T_1, \dotsc, T_k)$ of contractions with $\|q\|=1$ acting on a Hilbert space $\HS$. It is evident that  $(U_1^*, \dotsc, U_k^*)$ is a minimal $q$-unitary dilation of $(T_1^*, \dotsc, T_k^*)$. By regular $q$-unitary  dilation condition, we have 
	\begin{align}\label{eqn_102}
		\left[T_1^{*m_1^-}\dotsc T_k^{*m_k^-}\right]\left[T_1^{m_1^+} \dotsc T_k^{m_k^+}\right]=P_\mathcal{H}\left[U_1^{*m_1^-}\dotsc U_k^{*m_k^-}\right]\left[U_1^{m_1^+} \dotsc U_k^{m_k^+}\right]\bigg|_\mathcal{H}
	\end{align}
	for all $m_1, \dotsc, m_k \in \mathbb{Z}$. Clearly, $(-n)^-=n^+$ for all $n \in \Z$.  Let $m_1, \dotsc, m_k \in \Z$. Then
	\begin{align*}
		\ \left[T_1^{*m_1^+}\dotsc T_k^{*m_k^+}\right]\left[T_1^{m_1^-} \dotsc T_k^{m_k^-}\right]
		&=P_\mathcal{H}\left[U_1^{*m_1^+}\dotsc U_k^{*m_k^+}\right]\left[U_1^{m_1^-} \dotsc U_k^{m_k^-}\right]\bigg|_\mathcal{H} \qquad  [\text{by} \ \eqref{eqn_102}] \\
		&=\underset{1\leq i<j \leq k}{\prod}q_{ij}^{m_i^-m_j^+-m_i^+m_j^-} P_\mathcal{H}\left[U_1^{m_1^-}\dotsc U_k^{m_k^-}\right]\left[U_1^{*m_1^+} \dotsc U_k^{*m_k^+}\right]\bigg|_\mathcal{H},
	\end{align*}
	where the last equality follows from Lemma \ref{lem_107}. Thus, we obtain a generalization of \eqref{eqn_*reg2} in the $q$-commuting setting. This motivates us to put forth the following definition.
	
	\begin{defn}\label{defn_*-reg}
		Let $\underline{T}=(T_1, \dotsc, T_k)$ be a $q$-commuting tuple of contractions with $\|q\|=1$  acting on a Hilbert space $\HS$. A $q$-isometric dilation $\underline{V}=(V_1, \dotsc, V_k)$ is said to be \textit{$*$-regular}, if it satisfies the condition 
		\[
		\left[T_1^{m_1^+} \dotsc T_k^{m_k^+}\right]\left[T_1^{*m_1^-}\dotsc T_k^{*m_k^-}\right]=\underset{1\leq i<j \leq k}{\prod}q_{ij}^{m_i^-m_j^+-m_i^+m_j^-} P_\mathcal{H}\left[V_1^{*m_1^-}\dotsc V_k^{*m_k^-}\right]\left[V_1^{m_1^+} \dotsc V_k^{m_k^+}\right]\bigg|_\mathcal{H}
		\] 
		for all $m_1, \dotsc, m_k \in \mathbb{Z}$.  
	\end{defn}	
	
	From the discussion preceding Definition \ref{defn_*-reg}, we obtain the following lemma.
	
	\begin{lem}\label{lem_109}
		Let $\underline{T}=(T_1, \dotsc, T_k)$ be a $q$-commuting tuple of contractions with $\|q\|=1$ acting on a Hilbert space $\HS$. A $q$-commuting tuple of unitaries $\underline{U}=(U_1, \dotsc, U_k)$ is a minimal regular $q$-unitary  dilation of $\underline{T}$ if and only if $(U_1^*, \dotsc, U_k^*)$ is a minimal $*$-regular $q$-unitary dilation of $(T_1^*, \dotsc, T_k^*)$.
	\end{lem}
	
	The following lemmas summarize basic facts about $q$-commuting contractions with $\|q\|=1$.
	
	\begin{lem}\label{lem_101}
		A $q$-commuting pair of isometries $(V_1, V_2)$ with $\|q\|=1$  is doubly $q$-commuting if and only if $V_1 \text{ker}V_2^* \subseteq \text{ker}V_2^*$. 
	\end{lem}
	
	\begin{proof}
		Let $(V_1, V_2)$ be a doubly $q$-commuting pair of isometries acting on a Hilbert space $\HS$ and let $q_{12} \in \mathbb{T}$ be such that $V_1V_2=q_{12}V_2V_1$. If $V_1V_2^*=\overline{q}_{12}V_2^*V_1$, then $V_2^*V_1x=q_{12}V_1V_2^*x=0$ for all $x \in \text{ker}V_2^*$ and so, $V_1 \text{ker}V_2^* \subseteq \text{ker}V_2^*$. Conversely, assume that $V_1 \text{ker}V_2^* \subseteq \text{ker}V_2^*$. Since $V_2$ is an isometry, we can write $\HS=\text{ran}V_2 \oplus \text{ker}V_2^*$. Let $x \in \HS$. One can find $z_1 \in \HS$ and $x_1 \in \text{ker}V_2^*$ such that $x=V_2z_1+x_1$. Then
		\begin{align}\label{eqn_101}
			V_1V_2^*x=V_1V_2^*V_2z_1+V_1V_2^*x_1=V_1z_1,
		\end{align}
		where the last equality follows from the facts that $V_2$ is an isometry and $x_1 \in \text{ker}V_2^*$. Moreover, we have that $V_2^*V_1x_1=0$ and so,
		$
		V_2^*V_1x=V_2^*V_1V_2z_1+V_2^*V_1x_1=V_2^*(q_{12}V_2V_1)z_1=q_{12}V_1z_1=q_{12}V_1V_2^*x,
		$
		where the last equality follows from \eqref{eqn_101}. The proof is now complete.
	\end{proof}

	\begin{lem}\label{lem_106}
		Let $\underline{T}=(T_1, \dotsc, T_k)$ be a $q$-commuting tuple of contractions with $\|q\|=1$ acting on a Hilbert space $\HS$. If $\underline{V}=(V_1, \dotsc, V_k)$ is a minimal $q$-isometric dilation of $\underline{T}$, then $\HS$ is a joint invariant subspace for $V_1^*, \dotsc, V_k^*$ and $T_i^*=V_i^*|_{\HS}$ for $1\leq i \leq k$.
	\end{lem}
	
	\begin{proof}
		Let $\underline{V}=(V_1, \dotsc, V_k)$ be a minimal $q$-isometric dilation of $\underline{T}$ acting on the space $\mathcal{K}=\ov{\text{span}}\left\{V_1^{m_1}\dotsc V_k^{m_k}h : m_1, \dotsc, m_k \in \Z_+, h \in \HS \right\}$. Take $h \in \HS$ and $m_1, \dotsc, m_k \in \Z_+$. Using the $q$-commutativity hypothesis, we have that
		\begin{align*}
			T_iP_{\HS}V_1^{m_1}\dotsc V_k^{m_k}h
			=T_i T_1^{m_1}\dotsc T_k^{m_k}h
			&=\underset{1 \leq j<i}{\prod}q_{ij}^{m_j} \ T_1^{m_1}\dotsc T_{i-1}^{m_{i-1}}T_i^{m_i+1}T_{i+1}^{m_{i+1}}\dotsc T_k^{m_k}h\\
			&=\underset{1 \leq j<i}{\prod}q_{ij}^{m_j} P_{\HS} V_1^{m_1}\dotsc V_{i-1}^{m_{i-1}}V_i^{m_i+1}V_{i+1}^{m_{i+1}}\dotsc V_k^{m_k}h\\
			&=\underset{1 \leq j<i}{\prod}q_{ij}^{m_j}\underset{1 \leq j<i}{\prod}q_{ij}^{-m_j} P_{\HS} V_i V_1^{m_1}\dotsc V_k^{m_k}h\\
			&=P_{\HS} V_i V_1^{m_1}\dotsc V_k^{m_k}h.
		\end{align*}
		Consequently, $T_iP_{\HS}=P_{\HS}V_i$ for $1 \leq i \leq k$. Let $h \in \HS$ and $x \in \mathcal{K}$. Then
		\begin{align*}
			\la T_i^*h, x\ra=\la T_i^*h, P_\HS x\ra=\la h, T_iP_\HS x \ra=\la h, P_\HS V_ix  \ra=\la h, V_ix \ra=\la V_i^*h, x\ra     
		\end{align*}
		and so, $T_i^*=V_i^*|_{\HS}$ for $1 \leq i \leq k$. The proof is now complete.
	\end{proof}

	\begin{defn}
		Let $\underline{V}=(V_1, \dotsc, V_k)$ be a $q$-commuting tuple of isometries  with $\|q\|=1$  acting on a Hilbert space $\mathcal{K}$. We say that a $q$-commuting tuple of unitaries $\underline{U}=(U_1, \dotsc, U_k)$ acting on a Hilbert space $\mathcal{K}_0 \supseteq \mathcal{K}$ is a \textit{$q$-unitary extension} of $\underline{V}$ if $\mathcal{K}$ is a joint invariant subspace for $\underline{U}$ and $V_i=U_i|_{\mathcal{K}}$ for $1 \leq i \leq k$. The extension is minimal if 
		$
		\mathcal{K}_0=\ov{\text{span}}\left\{U_1^{*m_1}\dotsc U_k^{*m_k}x : m_1, \dotsc, m_k \in \Z_+, x \in \mathcal{K} \right\}.
		$
	\end{defn}
	
	We have by Theorem \ref{thm_001} that a $q$-commuting tuple of isometries $\underline{V}$ with $\|q\|=1$ admits a minimal $q$-unitary extension. It is not difficult to see that a minimal $q$-unitary extension of $\underline{V}$ is indeed a minimal $q$-unitary dilation of $\underline{V}$ and vice-versa.

	\begin{lem}\label{lem_208}
		Let $\underline{T}=(T_1, \dotsc, T_k)$ be a $q$-commuting tuple of contractions with $\|q\|=1$  acting on a Hilbert space $\HS$. If $\underline{V}$ is a minimal regular (respectively $*$-regular) $q$-isometric  dilation of $\underline{T}$, then the minimal $q$-unitary extension of $\underline{V}$ is a minimal regular (respectively $*$-regular) $q$-unitary dilation of $\underline{T}$.
	\end{lem}
	
	\begin{proof}
		Let $\underline{V}=(V_1, \dotsc, V_k)$ be a minimal regular $q$-isometric dilation of $\underline{T}$ acting on the space  
		\[
		\mathcal{K}_+=\ov{\text{span}}\left\{V_1^{m_1}\dotsc V_k^{m_k}h : m_1, \dotsc, m_k \in \Z_+, \ h \in \HS \right\}.
		\]
		Suppose $\underline{U}=(U_1, \dotsc, U_k)$ is a minimal $q$-unitary extension of $\underline{V}$ acting on the space $\mathcal{K}_0$ given by
		\[
		\mathcal{K}_0=\ov{\text{span}}\left\{U_1^{*m_1}\dotsc U_k^{*m_k}x \ : \ m_1, \dotsc, m_k \in \Z_+, \ x \in \mathcal{K}_+ \right\}.
		\]
		For minimality, note that $\mathcal{K}_0=\ov{\text{span}}\left\{U_1^{m_1}\dotsc U_k^{m_k}h : m_1, \dotsc, m_k \in \Z, \ h \in \HS \right\}$. We have 
		\begin{align}\label{eqn_103}
			U_k^{m_k^-}\dotsc U_1^{m_1^-}|_{\mathcal{K}_+}=V_k^{m_k^-}\dotsc V_1^{m_1^-} \quad \text{and so,} \quad P_{\mathcal{K}_+}U_1^{*m_1^-}\dotsc U_k^{*m_k^-}|_{\mathcal{K}_+}=V_1^{*m_1^-}\dotsc V_k^{*m_k^-},
		\end{align}
		for all $m_1, \dotsc, m_k \in \Z$. Here, $P_{\mathcal{K}_+}$ is the orthogonal projection of $\mathcal{K}_0$ onto $\mathcal{K}_+$. Since $\underline{V}$ is a minimal regular $q$-isometric dilation of $\underline{T}$, we have that
		\begin{align*}
			\left[T_1^{*m_1^-}\dotsc T_k^{*m_k^-}\right]\left[T_1^{m_1^+} \dotsc T_k^{m_k^+}\right]
			&=P_\mathcal{H}\left[V_1^{*m_1^-}\dotsc V_k^{*m_k^-}\right]\left[V_1^{m_1^+} \dotsc V_k^{m_k^+}\right]\bigg|_\mathcal{H}\\
			&=P_\mathcal{H}\left[U_1^{*m_1^-}\dotsc U_k^{*m_k^-}\right]\left[U_1^{m_1^+} \dotsc U_k^{m_k^+}\right]\bigg|_\mathcal{H},
		\end{align*}
		where the last equality follows from \eqref{eqn_103} and the fact that $V_i=U_i|_{\mathcal{K}_+}$ for $1 \leq i \leq k$. So, $\underline{U}$ is a minimal regular $q$-unitary dilation of $\underline{T}$. Similarly, one can prove the rest of the conclusion.
	\end{proof}
	
	Following the similar method as above, one can easily arrive at the following result. 
	
	\begin{cor}\label{cor_112}
		Let $\underline{T}=(T_1, \dotsc, T_k)$ be a $q$-commuting tuple of contractions with $\|q\|=1$  acting on a Hilbert space $\HS$. If $\underline{U}=(U_1, \dotsc, U_k)$ is a minimal regular $q$-unitary dilation of $\underline{T}$, then the restriction of $\underline{U}$ to the space 
		\[
		\mathcal{K}_+=\ov{\text{span}}\left\{U_1^{m_1}\dotsc U_k^{m_k}h : m_1, \dotsc, m_k \in \Z_+, h \in \HS \right\}
		\] 
		is a minimal regular $q$-isometric dilation of $\underline{T}$. Also, the restriction of $(U_1^*, \dotsc, U_k^*)$ to the space 
		\[
		\mathcal{K}_-=\ov{\text{span}}\left\{U_1^{*m_1}\dotsc U_k^{*m_k}h : m_1, \dotsc, m_k \in \Z_+, h \in \HS \right\}
		\] 
		is a minimal $*$-regular $q$-isometric dilation of $(T_1^*, \dotsc, T_k^*)$.
	\end{cor}
	
From here onwards, for a tuple $\underline{T}=(T_1, \dotsc, T_k)$ and $m=(m_1, \dotsc, m_k) \in \Z_+^k$, we denote by 
	\[
	\underline{T}^m=T_1^{m_1} \dotsc T_k^{m_k} \quad \text{and} \quad \underline{T}^{*m}=(T_1^{m_1} \dotsc T_k^{m_k})^*.
	\]
	For $1 \leq \alpha \leq k$, we denote by $e_\alpha$ the vector in $\Z_+^k$ with its $\alpha$-th entry $1$ and all other entries are zero. For $m=(m_1, \dotsc, m_k), n=(n_1, \dotsc, n_k) \in \Z^k$, we say that $m \geq n$ if $m_i \geq n_i$ for  $1 \leq i \leq k$. The following lemma is technical in nature but it plays a crucial role in establishing our main results.
	
	\begin{lem}\label{lem_112}
		Assume that $\underline{T}=(T_1, \dotsc, T_k)$ is a $q$-commuting tuple of contractions with $\|q\|=1$  acting on a Hilbert space $\HS$ and $\underline{V}=(V_1, \dotsc, V_k)$ is a minimal $*$-regular $q$-isometric dilation of $\underline{T}$. Let $1\leq \alpha \leq k$. For $m=(m_1, \dotsc, m_k), \mu=(\mu_1, \dotsc, \mu_k) \in \Z_+^k$ with $m_\alpha=0$ and $\mu_\alpha \geq 1$, we have 
		\smallskip
		\begin{align*}
			\la \underline{V}^mx, \underline{V}^\mu z\ra =\la \underline{V}^mV_\alpha T_\alpha^*x, \ \underline{V}^\mu z \ra 
		\end{align*}
		for every $x, z \in \HS$.
	\end{lem}
	
	\begin{proof}
		Assume that $T_iT_j=q_{ij}T_jT_i$ and $V_iV_j=q_{ij}V_jV_i$, where $q_{ij} \in \T$ for $1 \leq i, j \leq k$ with $i \ne j$. Choose any $\alpha \in \{1, \dotsc, k\}$ and $x, z \in \HS$. Let $m=(m_1, \dotsc, m_k), \mu=(\mu_1, \dotsc, \mu_k) \in \Z_+^k$ be such that $m_\alpha=0$ and $\mu_\alpha \geq 1$. We show that 
		\begin{align}\label{eqn_104}
			\underline{V}^{*\mu}\underline{V}^{m}=\underset{1 \leq i<j \leq k}{\prod}q_{ij}^{\alpha_{ij}(m, \mu)}\underline{V}^{*(m-\mu)^-}\underline{V}^{(m-\mu)^+},
		\end{align}
		where $\alpha_{ij}(m, \mu)=(m_i-\mu_i)(m_j-(m_j-\mu_j)^+)$ for $1 \leq i, j \leq k$ with $i<j$. 
		Since  $\underline{V}$ consists of isometries, we have 
		\begin{align}\label{eqn_105}
			V_j^{*\mu_j}V_j^{m_j}=V_j^{*(m_j-\mu_j)^-}V_j^{(m_j-\mu_j)^+}
		\end{align}
		for $1 \leq j \leq k$. It follows from the $q$-commutativity hypothesis  that 
		\begin{align}\label{eqn_106}
			V_i^{*n_i}V_j^{*n_j}=q_{ij}^{n_in_j}V_j^{*n_j}V_i^{*n_i} \qquad (n_i, n_j \in \Z_+)
		\end{align} 
		for $1 \leq i, j \leq k$ with $i \ne j$. Then 
		\begin{align*}
			V_2^{*\mu_2}V_1^{*\mu_1}V_1^{m_1}V_2^{m_2}
			&=V_2^{*\mu_2}V_1^{*(m_1-\mu_1)^-}V_1^{(m_1-\mu_1)^+}V_2^{m_2} \quad [\text{by \eqref{eqn_105}}]\\
			&=q_{12}^{m_2(m_1-\mu_1)^+-\mu_2(m_1-\mu_1)^-}V_1^{*(m_1-\mu_1)^-}V_2^{*\mu_2}V_2^{m_2}V_1^{(m_1-\mu_1)^+}\quad [\text{by \eqref{eqn_106}}]\\
			&=q_{12}^{m_2(m_1-\mu_1)^+-\mu_2(m_1-\mu_1)^-}V_1^{*(m_1-\mu_1)^-}V_2^{*(m_2-\mu_2)^-}V_2^{(m_2-\mu_2)^+}V_1^{(m_1-\mu_1)^+}\quad [\text{by \eqref{eqn_105}}]\\
			&=q_{12}^{m_2(m_1-\mu_1)^+-\mu_2(m_1-\mu_1)^-+(m_1-\mu_1)^-(m_2-\mu_2)^--(m_2-\mu_2)^+(m_1-\mu_1)^+} \\ 
			& \quad \
			V_2^{*(m_2-\mu_2)^-}V_1^{*(m_1-\mu_1)^-}V_1^{(m_1-\mu_1)^+}V_2^{(m_2-\mu_2)^+}\quad [\text{by \eqref{eqn_106}}].
		\end{align*}
		The exponent of $q_{12}$ in the above equation is given by
		\begin{align*}
			&\quad \ m_2(m_1-\mu_1)^+-\mu_2(m_1-\mu_1)^-+(m_1-\mu_1)^-(m_2-\mu_2)^--(m_2-\mu_2)^+(m_1-\mu_1)^+\\
			&=(m_1-\mu_1)^+(m_2-(m_2-\mu_2)^+)-(m_1-\mu_1)^-(\mu_2-(m_2-\mu_2)^-)\\
			&=(m_2-(m_2-\mu_2)^+)((m_1-\mu_1)^+-(m_1-\mu_1)^-) \quad [\text{as} \ m_2-(m_2-\mu_2)^+=\mu_2-(m_2-\mu_2)^- ]\\
			&=(m_2-(m_2-\mu_2)^+)(m_1-\mu_1)\\
			&=\alpha_{12}(m, \mu),
		\end{align*}
		where $\alpha_{12}(m, \mu)$ is as in \eqref{eqn_104}. Putting everything together, we have that 
		\[	V_2^{*\mu_2}V_1^{*\mu_1}V_1^{m_1}V_2^{m_2}=q_{12}^{\alpha_{12}(m, \mu)}V_2^{*(m_2-\mu_2)^-}V_1^{*(m_1-\mu_1)^-}V_1^{(m_1-\mu_1)^+}V_2^{(m_2-\mu_2)^+},
		\]
		which gives \eqref{eqn_104} when $k=2$. The general case follows using induction establishing \eqref{eqn_104}. A repeated application of $q$-commutativity condition as in \eqref{eqn_106} gives that
		
		\[
		\underline{V}^{*(m-\mu)^-}=V_k^{*(m_k-\mu_k)^-}\dotsc V_1^{*(m_1-\mu_1)^-}=\underset{1 \leq i<j \leq k}{\prod}q_{ij}^{-\beta_{ij}(m, \mu)}V_1^{*(m_k-\mu_k)^-}\dotsc V_k^{*(m_k-\mu_k)^-},
		\]
		where $\beta_{ij}(m, \mu)=(m_i-\mu_i)^-(m_j-\mu_j)^-$ for $1 \leq i, j \leq k$ with $i<j$. Combining above equation with \eqref{eqn_104}, we have
		\begin{align}\label{eqn_107}
			\underline{V}^{*\mu}\underline{V}^{m}=\underset{1 \leq i<j \leq k}{\prod}q_{ij}^{\gamma_{ij}(m, \mu)} \ \left[V_1^{*(m_1-\mu_1)^-}\dotsc V_k^{*(m_k-\mu_k)^-}\right]\left[V_1^{(m_1-\mu_1)^+} \dotsc V_k^{(m_k-\mu_k)^+}\right],
		\end{align}
		where $\gamma_{ij}(m, \mu)=\alpha_{ij}(m, \mu)-\beta_{ij}(m, \mu)$. Since $\underline{V}$ is a $*$-regular dilation of $\underline{T}$, we have that 
		\begin{align}\label{eqn_108}
			&\quad \ \la \underline{V}^mx, \underline{V}^{\mu}z\ra \notag \\
			&= \la \underline{V}^{*\mu}\underline{V}^mx, z\ra \notag \\
			&= \underset{1 \leq i<j \leq k}{\prod}q_{ij}^{\gamma_{ij}(m, \mu)} \la \left[V_1^{*(m_1-\mu_1)^-}\dotsc V_k^{*(m_k-\mu_k)^-}\right]\left[V_1^{(m_1-\mu_1)^+} \dotsc V_k^{(m_k-\mu_k)^+}\right]x, z\ra \quad [\text{by \eqref{eqn_107}}] \notag \\
			&= \underset{1 \leq i<j \leq k}{\prod}q_{ij}^{\gamma_{ij}(m, \mu)} \underset{1 \leq i<j \leq k}{\prod}q_{ij}^{-t_{ij}(m, \mu)} \la \left[T_1^{(m_1-\mu_1)^+} \dotsc T_k^{(m_k-\mu_k)^+}\right]\left[T_1^{*(m_k-\mu_k)^-}\dotsc T_k^{*(m_k-\mu_k)^-}\right]x, z\ra,
		\end{align}
		where 
		\[
		t_{ij}(m, \mu)=(m_i-\mu_i)^-(m_j-\mu_j)^+-(m_i-\mu_i)^+(m_j-\mu_j)^-
		\] 
		for $1 \leq i, j \leq k$ with $i<j$. We define the scalars $s_{ij}(m, \mu)$ for $1 \leq i, j \leq k$ with $i<j$ as follows:
		\[
		s_{ij}(m, \mu)= \left\{
		\begin{array}{ll}
			t_{ij}(m, \mu) & i \ne \alpha, j \ne \alpha, \\
			(m_i-\mu_i)^-(m_\alpha+1-\mu_\alpha)^+-(m_i-\mu_i)^+(m_\alpha+1-\mu_\alpha)^- & i<\alpha, j=\alpha,\\
			(m_\alpha+1-\mu_\alpha)^-(m_j-\mu_j)^+-(m_\alpha+1-\mu_\alpha)^+(m_j-\mu_j)^- & i=\alpha, \alpha<j.\\
		\end{array} 
		\right. 
		\]
		Since $\underline{V}$ is $*$-regular, we have
		\begin{align}\label{eqn_109}
			& \la \left[V_1^{*(m_1-\mu_1)^-}\dotsc V_\alpha^{*(m_\alpha+1-\mu_\alpha)^-} \dotsc V_k^{*(m_k-\mu_k)^-}\right]\left[V_1^{(m_1-\mu_1)^+} \dotsc V_\alpha^{(m_\alpha+1-\mu_\alpha)^+} \dotsc  V_k^{(m_k-\mu_k)^+}\right]
			T_\alpha^*x, z  \ra \notag \\
			&= \underset{1 \leq i<j \leq k}{\prod}q_{ij}^{-s_{ij}(m, \mu)} \notag \\
			& \ \la \left[T_1^{(m_1-\mu_1)^+} \dotsc T_\alpha^{(m_\alpha+1-\mu_\alpha)^+} \dotsc  T_k^{(m_k-\mu_k)^+}\right] \left[T_1^{*(m_1-\mu_1)^-}\dotsc T_\alpha^{*(m_\alpha+1-\mu_\alpha)^-} \dotsc T_k^{*(m_k-\mu_k)^-}\right] 
			T_\alpha^*x, z  \ra.
		\end{align}
		From the conditions $m_\alpha=0$ and $\mu_\alpha \geq 1$, it follows that
		\begin{align}\label{eqn_110}
			(m_\alpha+1-\mu_\alpha)^+=(m_\alpha-\mu_\alpha)^+ \qquad \text{and} \qquad 	(m_\alpha+1-\mu_\alpha)^-+1=(m_\alpha-\mu_\alpha)^-.
		\end{align}
		Some routine computations give that
		\begin{align*}
			\alpha_{ij}(m, \mu)-\alpha_{ij}(m+e_\alpha, \mu)
			= \left\{
			\begin{array}{ll}
				0 & i \ne \alpha, j \ne \alpha, \\
				\mu_i-m_i & i<\alpha, j=\alpha,\\
			(	m_j-\mu_j)^+-m_j & i=\alpha, \alpha<j\\
			\end{array} 
			\right. 
		\end{align*}
	and 	
		\begin{align*}
			\beta_{ij}(m, \mu)-\beta_{ij}(m+e_\alpha, \mu)
			= \left\{
			\begin{array}{ll}
				0 & i \ne \alpha, j \ne \alpha, \\
				(m_i-\mu_i)^- & i<\alpha, j=\alpha,\\
				(m_j-\mu_j)^- & i=\alpha, \alpha<j.\\
			\end{array} 
			\right. 
		\end{align*}
			Using \eqref{eqn_110} along with above computations, it follows that
		\begin{align}\label{eqn_115}
			\gamma_{ij}(m, \mu)-\gamma_{ij}(m+e_\alpha, \mu)
			&=(\alpha_{ij}(m, \mu)-\alpha_{ij}(m+e_\alpha, \mu))-(\beta_{ij}(m, \mu)-\beta_{ij}(m+e_\alpha, \mu)) \notag \\
			&= \left\{
			\begin{array}{ll}
				0 & i \ne \alpha, j \ne \alpha, \\
				-(m_i-\mu_i)^+ & i<\alpha, j=\alpha,\\
				-\mu_j & i=\alpha, \alpha<j.\\
			\end{array} 
			\right. 
		\end{align}
	A few steps of calculations give that
		\begin{align}\label{eqn_111}
			&  \quad \la \underline{V}^mV_\alpha T_\alpha^*x, \ \underline{V}^\mu z \ra \notag \\
				 &=\la V_k^{*\mu_k} \dotsc V_1^{*\mu_1}V_1^{m_1}\dotsc V_k^{m_k}V_\alpha T_\alpha^*x, z  \ra \notag  \\
			&=\underset{j:\alpha<j}{\prod}q_{j\alpha}^{m_j}\la V_k^{*\mu_k} \dotsc V_1^{*\mu_1}V_1^{m_1}\dotsc V_\alpha^{m_\alpha+1} \dotsc V_k^{m_k}T_\alpha^*x, z  \ra \notag \\
			&=\underset{j:\alpha<j}{\prod}q_{\alpha j}^{-m_j}\la \underline{V}^{*\mu} \underline{V}^{m+e_\alpha} T_\alpha^*x, z  \ra \notag  \\
			&=\underset{j:\alpha<j}{\prod}q_{\alpha j}^{-m_j}\underset{1 \leq i<j \leq k}{\prod}q_{ij}^{\gamma_{ij}(m+e_\alpha, \mu)} \qquad \qquad \qquad \qquad \qquad \quad \quad \qquad \qquad \quad \quad \quad \quad \quad \quad  [\text{by \eqref{eqn_107}}]  \notag \\
			& \ \la \left[V_1^{*(m_1-\mu_1)^-}\dotsc V_\alpha^{*(m_\alpha+1-\mu_\alpha)^-} \dotsc V_k^{*(m_k-\mu_k)^-}\right]\left[V_1^{(m_1-\mu_1)^+} \dotsc V_\alpha^{(m_\alpha+1-\mu_\alpha)^+} \dotsc  V_k^{(m_k-\mu_k)^+}\right]
			T_\alpha^*x, z  \ra \notag \\
			&= \underset{j:\alpha<j}{\prod}q_{\alpha j}^{-m_j}\underset{1 \leq i<j \leq k}{\prod}q_{ij}^{\gamma_{ij}(m+e_\alpha, \mu)} \underset{1 \leq i<j \leq k}{\prod}q_{ij}^{-s_{ij}(m, \mu)}  \qquad \qquad \qquad \qquad \qquad \quad \quad \quad \quad \quad [\text{by \eqref{eqn_109}}] \notag  \\
			& \ \la \left[T_1^{(m_1-\mu_1)^+} \dotsc T_\alpha^{(m_\alpha+1-\mu_\alpha)^+} \dotsc  T_k^{(m_k-\mu_k)^+}\right] \left[T_1^{*(m_1-\mu_1)^-}\dotsc T_\alpha^{*(m_\alpha+1-\mu_\alpha)^-} \dotsc T_k^{*(m_k-\mu_k)^-}\right]
			T_\alpha^*x, z  \ra \notag  \\
			&= \underset{j:\alpha<j}{\prod}q_{\alpha j}^{-m_j}\underset{1 \leq i<j \leq k}{\prod}q_{ij}^{\gamma_{ij}(m+e_\alpha, \mu)} \underset{1 \leq i<j \leq k}{\prod}q_{ij}^{-s_{ij}(m, \mu)}  \underset{j:\alpha<j}{\prod}q_{\alpha j}^{-(m_j-\mu_j)^-} \quad \quad \qquad [\text{as $\underline{T}$ is $q$-commuting} ] \notag \\
			& \ \la \left[T_1^{(m_1-\mu_1)^+} \dotsc T_\alpha^{(m_\alpha+1-\mu_\alpha)^+} \dotsc  T_k^{(m_k-\mu_k)^+}\right] \left[T_1^{*(m_1-\mu_1)^-}\dotsc T_\alpha^{*((m_\alpha+1-\mu_\alpha)^-+1)} \dotsc T_k^{*(m_k-\mu_k)^-}\right]
			x, z  \ra \notag  \\
			&= \underset{1 \leq i<j \leq k}{\prod}q_{ij}^{\gamma_{ij}(m+e_\alpha, \mu)} \underset{j:\alpha<j}{\prod}q_{\alpha j}^{-m_j-(m_j-\mu_j)^-} \underset{1 \leq i<j \leq k}{\prod}q_{ij}^{-s_{ij}(m, \mu)} \notag \\
			& \ \la \left[T_1^{(m_1-\mu_1)^+} \dotsc T_\alpha^{(m_\alpha-\mu_\alpha)^+} \dotsc  T_k^{(m_k-\mu_k)^+}\right] \left[T_1^{*(m_1-\mu_1)^-}\dotsc T_\alpha^{*(m_\alpha-\mu_\alpha)^-} \dotsc T_k^{*(m_k-\mu_k)^-}\right] 
			x, z  \ra,
		\end{align}
		where the last equality follows from \eqref{eqn_110}. Consequently, it follows from \eqref{eqn_108} and \eqref{eqn_111} that
		\begin{align}\label{eqn_112}
			& \quad \ \la \underline{V}^mV_\alpha T_\alpha^*x, \ \underline{V}^\mu z \ra \notag \\	
			&= \underset{1 \leq i<j \leq k}{\prod}q_{ij}^{\gamma_{ij}(m+e_\alpha, \mu)} \underset{j:\alpha<j}{\prod}q_{\alpha j}^{-m_j-(m_j-\mu_j)^-} \underset{1 \leq i<j \leq k}{\prod}q_{ij}^{-s_{ij}(m, \mu)} \underset{1 \leq i<j \leq k}{\prod}q_{ij}^{-\gamma_{ij}(m, \mu)+t_{ij}(m, \mu)}\la \underline{V}^m x, \ \underline{V}^\mu z \ra. \notag \\
		\end{align}
		We show that
		\begin{align}\label{eqn_113}
			\underset{1 \leq i<j \leq k}{\prod}q_{ij}^{\gamma_{ij}(m+e_\alpha, \mu)-\gamma_{ij}(m, \mu)} \underset{j:\alpha<j}{\prod}q_{\alpha j}^{-m_j-(m_j-\mu_j)^-} \underset{1 \leq i<j \leq k}{\prod}q_{ij}^{t_{ij}(m, \mu)-s_{ij}(m, \mu)}=1.
		\end{align}
		By definition of $t_{ij}(m, \mu)$ and $s_{ij}(m, \mu)$, we have that
		\begin{align}\label{eqn_114}
			\underset{1 \leq i<j \leq k}{\prod}q_{ij}^{t_{ij}(m, \mu)-s_{ij}(m, \mu)}
			&= \prod_{\substack{1 \leq i<j \leq k  \\ i \ne \alpha, j \ne \alpha}}q_{ij}^{t_{ij}(m, \mu)-s_{ij}(m, \mu)} \prod_{\substack{1 \leq i<j \leq k  \\ i \ne \alpha, j = \alpha}}q_{ij}^{t_{ij}(m, \mu)-s_{ij}(m, \mu)} \prod_{\substack{1 \leq i<j \leq k  \\ i = \alpha, j \ne \alpha}}q_{ij}^{t_{ij}(m, \mu)-s_{ij}(m, \mu)} \notag \\ 
			&= \prod_{\substack{i: i< \alpha}}q_{i\alpha}^{t_{ij}(m, \mu)-s_{ij}(m, \mu)} \prod_{\substack{j: \alpha<j}}q_{\alpha j}^{t_{ij}(m, \mu)-s_{ij}(m, \mu)}\notag \\
			&= \prod_{\substack{i: i< \alpha}}q_{i\alpha}^{(m_i-\mu_i)^-\left[(m_\alpha-\mu_\alpha)^+-(m_\alpha+1-\mu_\alpha)^+\right]-(m_i-\mu_i)^+\left[(m_\alpha-\mu_\alpha)^--(m_\alpha+1-\mu_\alpha)^-\right]} \notag \\
			& \quad  \prod_{\substack{j: \alpha<j}}q_{\alpha j}^{(m_j-\mu_j)^+\left[(m_\alpha-\mu_\alpha)^--(m_\alpha+1-\mu_\alpha)^-\right]-(m_j-\mu_j)^-\left[(m_\alpha-\mu_\alpha)^+-(m_\alpha+1-\mu_\alpha)^+\right]} \notag \\
			&= \prod_{\substack{i: i< \alpha}}q_{i\alpha}^{-(m_i-\mu_i)^+}   \prod_{\substack{j: \alpha<j}}q_{\alpha j}^{(m_j-\mu_j)^+} \qquad [\text{by \eqref{eqn_110}}].\notag\\
		\end{align} 
		Consequently, we have by \eqref{eqn_114} and \eqref{eqn_115}  that
		\begin{align*}
			&=\left(\underset{1 \leq i<j \leq k}{\prod}q_{ij}^{\gamma_{ij}(m+e_\alpha, \mu)-\gamma_{ij}(m, \mu)}\right) \left(\underset{j:\alpha<j}{\prod}q_{\alpha j}^{-m_j-(m_j-\mu_j)^-}\right) \left(\underset{1 \leq i<j \leq k}{\prod}q_{ij}^{t_{ij}(m, \mu)-s_{ij}(m, \mu)}\right)\\
			&= \left(\prod_{\substack{i: i< \alpha}}q_{i\alpha}^{(m_i-\mu_i)^+}   \prod_{\substack{j: \alpha<j}}q_{\alpha j}^{\mu_j}\right)\left(\underset{j:\alpha<j}{\prod}q_{\alpha j}^{-m_j-(m_j-\mu_j)^-}\right)\left( \prod_{\substack{i: i< \alpha}}q_{i\alpha}^{-(m_i-\mu_i)^+}   \prod_{\substack{j: \alpha<j}}q_{\alpha j}^{(m_j-\mu_j)^+}\right)\\
			&=\prod_{\substack{j: \alpha<j}}q_{\alpha j}^{(m_j-\mu_j)^+-(m_j-\mu_j)^--(m_j-\mu_j)}\\
			&=1,
		\end{align*}
		where the last equality follows from the fact that $n=n^+-n^-$ for all $n \in \Z$. Therefore, the equality in \eqref{eqn_113} holds, and the desired conclusion follows from \eqref{eqn_112}.
	\end{proof}
	
	We now present one of the main results of this article, which generalizes Theorem \ref{thm_104}, originally due to Ga\c spar–Suciu \cite{Gaspar} and Timotin \cite{Timotin}.
	
	\begin{thm}\label{thm_mainI}
		Let $\underline{T}=(T_1, \dotsc, T_k)$ be a $q$-commuting tuple of contractions with $\|q\|=1$ acting on a Hilbert space $\HS$ and let $\underline{V}=(V_1, \dotsc, V_k)$ be a minimal $q$-isometric dilation of $\underline{T}$. Then $\underline{V}$ is $*$-regular if and only if $\underline{V}$ is doubly $q$-commuting.
	\end{thm}
	
	\begin{proof}
		Suppose $T_iT_j=q_{ij}T_jT_i$ and $V_iV_j=q_{ij}V_jV_i$, where $q_{ij} \in \T$ for $1 \leq i, j \leq k$ with $i < j$. Assume that $\underline{V}=(V_1, \dotsc, V_k)$ is a minimal $*$-regular $q$-isometric dilation of $\underline{T}$ acting on a space $\mathcal{K}$.  Then
		\[
		\mathcal{K}=\ov{\text{span}}\left\{V_1^{m_1}\dotsc V_k^{m_k}x : m_1, \dotsc, m_k \in \Z_+, \ x \in \HS \right\}.
		\]
		Choose $\alpha, \beta \in \{1, \dotsc, k\}$ with $\alpha \ne \beta$. It is evident that
		\[
		V_\alpha \mathcal{K}=\ov{\text{span}}\left\{V_\alpha V_1^{m_1}\dotsc V_k^{m_k}x : m_1, \dotsc, m_k \in \Z_+, \ x \in \HS \right\}=\ov{\text{span}}\left\{\underline{V}^mx: x \in \HS, m \in \Z_+^k, m \geq e_\alpha\right\}.
		\]
		It is not difficult to see that 
		\[
		\text{ker}V_\alpha^*=\mathcal{K} \ominus V_\alpha \mathcal{K}=\ov{\text{span}}\left\{P_{\text{ker}V_\alpha^*} \ V_1^{m_1}\dotsc V_k^{m_k}x : m_1, \dotsc, m_k \in \Z_+,\  m_\alpha=0, \ x \in \HS \right\}.
		\]
		Fix $x \in \HS$ and define $y_m=\underline{V}^m(x-V_\alpha T_\alpha^*x)$ for $m=(m_1, \dotsc, m_k) \in \Z_+^k$ with $m_\alpha=0$. Let $z \in \HS$ and $\mu=(\mu_1, \dotsc, \mu_k) \in \Z_+^k$ with $\mu_\alpha \geq 1$. It follows from Lemma \ref{lem_112} that
		\[
		\la y_m, \underline{V}^\mu z \ra=\la \underline{V}^mx-\underline{V}^mV_\alpha T_\alpha^*x, \ \underline{V}^\mu z \ra=\la \underline{V}^mx, \underline{V}^\mu z\ra -\la \underline{V}^mV_\alpha T_\alpha^*x, \ \underline{V}^\mu z \ra=0
		\]
		and so, $y_m \in \mathcal{K} \ominus V_\alpha \mathcal{K}=\text{ker}V_\alpha^*$. Also, note that
		\[
		\underline{V}^mV_\alpha T_\alpha^*x=V_1^{m_1}\dotsc V_\alpha^{m_\alpha}\dotsc V_k^{m_k}V_\alpha T_\alpha^*x=\prod_{\substack{j: \alpha<j}}q_{\alpha j}^{-m_j}\left(V_1^{m_1}\dotsc V_\alpha^{m_\alpha+1}\dotsc V_k^{m_k}\right)T_\alpha^*x \in V_\alpha \mathcal{K}
		\]
		and thus, $P_{\text{ker}V_\alpha^*}(\underline{V}^mV_\alpha T_\alpha^*x)=0$. Consequently, we have that 
		\begin{align}\label{eqn_116}
			\underline{V}^m(x-V_\alpha T_\alpha^*x)=y_m=P_{\text{ker}V_\alpha^*}y_m=P_{\text{ker}V_\alpha^*} (\underline{V}^mx)-P_{\text{ker}V_\alpha^*}(\underline{V}^mV_\alpha T_\alpha^*x)=P_{\text{ker}V_\alpha^*} (\underline{V}^mx)
		\end{align}
		for all $x \in \HS$ and $m=(m_1, \dotsc, m_k) \in \Z_+^k$ with $m_\alpha=0$. This shows that
		\begin{align}\label{eqn_117}
			\text{ker} V_\alpha^*=\ov{\text{span}}\left\{\underline{V}^m(x-V_\alpha T_\alpha^*x) : m=(m_1, \dotsc, m_k) \in \Z_+^k,\  m_\alpha=0, \ x \in \HS \right\}.
		\end{align}
		Let $x \in \HS$ and $m=(m_1, \dotsc, m_k) \in \Z_+^k$ with $m_\alpha=0$. Since $\alpha \ne \beta$, we have that $(m + e_\beta)_\alpha = m_\alpha = 0$, where $(m + e_\beta)_\alpha$ denotes the $\alpha$-th component of $(m + e_\beta)$. It follows that
		\begin{align*}
			V_\beta \underline{V}^m(x-V_\alpha T_\alpha^*x) 
			=V_\beta V_1^{m_1}\dotsc V_\beta^{m_\beta}\dotsc V_k^{m_k}(x-V_\alpha T_\alpha^*x)
			&=\underset{j: \beta<j}{\prod}q_{\beta j}^{m_j}V_1^{m_1}\dotsc V_\beta^{m_\beta+1}\dotsc V_k^{m_k}(x-V_\alpha T_\alpha^*x)\\
			&=\underset{j: \beta<j}{\prod}q_{\beta j}^{m_j} \ \underline{V}^{m+e_\beta}(x-V_\alpha T_\alpha^*x)\\
			&=\underset{j: \beta<j}{\prod}q_{\beta j}^{m_j} \ P_{\text{ker}V_\alpha^*}(\underline{V}^{m+e_\beta}x) \qquad [\text{by \eqref{eqn_116}}]
		\end{align*}
		and thus, $V_\beta \underline{V}^m(x-V_\alpha T_\alpha^*x) \in \text{ker} V_\alpha^*$. So, we have by \eqref{eqn_117} that $V_\beta \text{ker} V_\alpha^* \subseteq \text{ker} V_\alpha^*$. It follows from Lemma \ref{lem_101} that $V_\alpha V_\beta^* =\ov{q}_{\alpha\beta}V_\beta^* V_\alpha$ and so, $\underline{V}$ is doubly $q$-commuting. 
		
		\vspace{0.2cm}
		
		To see the converse, assume that $V_iV_j=q_{ij}V_jV_i$ and $V_iV_j^*=\ov{q}_{ij}V_j^*V_i$, where $q_{ij} \in \T$ for $1 \leq i, j \leq k$ with $i \ne j$. Let $x, y \in \HS$ and let $m_1, \dotsc, m_k \in \Z$. Then
		\begin{align*}
			& \quad \ \la \left[V_1^{*m_1^-}\dotsc V_k^{*m_k^-}\right]\left[V_1^{m_1^+} \dotsc V_k^{m_k^+}\right]x, y \ra \\
			&=\underset{1\leq i<j \leq k}{\prod}q_{ij}^{-m_i^-m_j^++m_i^+m_j^-}\la  \left[V_1^{m_1^+} \dotsc V_k^{m_k^+}\right]\left[V_1^{*m_1^-}\dotsc V_k^{*m_k^-}\right]x, y \ra \quad [\text{by Lemma \ref{lem_107}} ]\\
			&=\underset{1\leq i<j \leq k}{\prod}q_{ij}^{-m_i^-m_j^++m_i^+m_j^-}\la  \left[V_1^{m_1^+} \dotsc V_k^{m_k^+}\right]\left[T_1^{*m_1^-}\dotsc T_k^{*m_k^-}\right]x, y \ra \quad [\text{by Lemma \ref{lem_106}} ]\\
			&=\underset{1\leq i<j \leq k}{\prod}q_{ij}^{-m_i^-m_j^++m_i^+m_j^-}\la  P_{\HS}\left[V_1^{m_1^+} \dotsc V_k^{m_k^+}\right]\left[T_1^{*m_1^-}\dotsc T_k^{*m_k^-}\right]x, y \ra \\
			&=\underset{1\leq i<j \leq k}{\prod}q_{ij}^{-m_i^-m_j^++m_i^+m_j^-}\la  \left[T_1^{m_1^+} \dotsc T_k^{m_k^+}\right]\left[T_1^{*m_1^-}\dotsc T_k^{*m_k^-}\right]x, y \ra,
		\end{align*}
		where the last equality follows from the fact that $\underline{V}$ is a $q$-isometric dilation of $\underline{T}$. Consequently, $\underline{V}$ is $*$-regular and the proof is now complete.
	\end{proof}
	
We now present a few applications of Theorem \ref{thm_mainI}. In this direction, our next result generalizes Proposition \ref{prop_105}, which is due to Ga\c spar–Suciu \cite{Gaspar}. Indeed, the next theorem characterizes when a $q$-commuting tuple of contractions with $\|q\|=1$ becomes a doubly $q$-commuting tuple. 
	
	\begin{thm}\label{thm_mainII}
		For a $q$-commuting tuple of contractions  $\underline{T}=(T_1, \dotsc, T_k)$ with $\|q\|=1$ acting on a Hilbert space $\HS$, the following statements are equivalent:
		\begin{enumerate}
			\item $\underline{T}$ is doubly $q$-commuting;
			\item $\underline{T}$ has a minimal regular $q$-isometric dilation which is doubly $q$-commuting;
			\item $\underline{T}$ has a minimal regular $q$-unitary dilation $\underline{U}=(U_1, \dotsc, U_k)$ such that $(U_1^*, \dotsc, U_k^*)$ is a  minimal regular $q$-unitary dilation $(T_1^*, \dotsc, T_k^*)$.
		\end{enumerate} 
	\end{thm}
	
	\begin{proof} We prove that $(1) \implies (3) \implies (2) \implies (1)$. 
		
		\vspace{0.2cm}
		
		\noindent $(1) \implies (3)$. Let $\underline{T}$ be doubly $q$-commuting. It follows from Theorem \ref{thm_001} that $\underline{T}$ admits a regular $q$-unitary dilation, say $\underline{N}=(N_1, \dotsc, N_k)$, acting on a Hilbert space $\mathcal{K}_0$. The space 
		\[
		\mathcal{K}=\ov{\text{span}}\left\{N_1^{m_1}\dotsc N_k^{m_k}h : m_1, \dotsc, m_k \in \Z, \ h \in \HS \right\}
		\]
		is a closed joint reducing subspace of $\underline{N}$. So, the tuple
		$
		\underline{U}=(U_1, \dotsc, U_k)=(N_1|_{\mathcal{K}}, \dotsc, N_k|_{\mathcal{K}})
		$	
		is a minimal regular $q$-unitary dilation of $\underline{T}$. Since $\underline{U}$ and $\underline{T}$ are doubly $q$-commuting tuples that satisfy the same $q$-intertwining relations, it follows from Lemma \ref{lem_107} that $(U_1^*, \dotsc, U_k^*)$ is a  minimal regular $q$-unitary dilation $(T_1^*, \dotsc, T_k^*)$. 
		
		\vspace{0.2cm}
		
		\noindent $(3) \implies (2)$. Suppose $\underline{T}$ has a minimal regular $q$-unitary dilation $\underline{U}=(U_1, \dotsc, U_k)$ acting on a Hilbert space $\mathcal{K}_0$ such that $(U_1^*, \dotsc, U_k^*)$ is a  minimal regular $q$-unitary dilation $(T_1^*, \dotsc, T_k^*)$. By minimality condition, we have that $\mathcal{K}_0=\ov{\text{span}}\left\{U_1^{m_1}\dotsc U_k^{m_k}h : m_1, \dotsc, m_k \in \Z, h \in \HS \right\}$. The space given by
		\[
		\mathcal{K}_+=\ov{\text{span}}\left\{U_1^{m_1}\dotsc U_k^{m_k}h : m_1, \dotsc, m_k \in \Z_+, h \in \HS \right\}
		\]
		is a closed joint invariant subspace of $\underline{U}$. It follows from Corollary \ref{cor_112} that the tuple
		\[
		\underline{V}=(V_1, \dotsc, V_k)=(U_1|_{{\mathcal{K}_+}}, \dotsc, U_k|_{{\mathcal{K}}_+})
		\]
		is a minimal regular $q$-isometric dilation of $\underline{T}$. Let $m_1, \dotsc, m_k \in \Z$. It is evident that
		\begin{align}\label{eqn_118}
			U_k^{m_k^-}\dotsc U_1^{m_1^-}|_{\mathcal{K}_+}=V_k^{m_k^-}\dotsc V_1^{m_1^-} \quad \text{and so,} \quad P_{\mathcal{K}_+}U_1^{*m_1^-}\dotsc U_k^{*m_k^-}|_{\mathcal{K}_+}=V_1^{*m_1^-}\dotsc V_k^{*m_k^-},
		\end{align}
		where $P_{\mathcal{K}_+}$ is the orthogonal projection of $\mathcal{K}_0$ onto $\mathcal{K}_+$. Since$(U_1^*, \dotsc, U_k^*)$ is a  minimal regular $q$-unitary dilation $(T_1^*, \dotsc, T_k^*)$, we have
		\begin{align*}
			& \quad \left[T_1^{m_1^+} \dotsc T_k^{m_k^+}\right]\left[T_1^{*m_1^-}\dotsc T_k^{*m_k^-}\right]\\
			&=P_\mathcal{H}\left[U_1^{m_1^+} \dotsc U_k^{m_k^+}\right]\left[U_1^{*m_1^-}\dotsc U_k^{*m_k^-}\right]\bigg|_\mathcal{H}\\
			&=\underset{1\leq i<j \leq k}{\prod}q_{ij}^{m_i^-m_j^+-m_i^+m_j^-}P_\HS \left[U_1^{*m_1^-}\dotsc U_k^{*m_k^-}\right]\left[U_1^{m_1^+} \dotsc U_k^{m_k^+}\right]\bigg|_{\HS} \quad [\text{by Lemma \ref{lem_107}}]\\
			&=\underset{1\leq i<j \leq k}{\prod}q_{ij}^{m_i^-m_j^+-m_i^+m_j^-}P_\HS \left[V_1^{*m_1^-}\dotsc V_k^{*m_k^-}\right]\left[V_1^{m_1^+} \dotsc V_k^{m_k^+}\right]\bigg|_{\HS} \quad [\text{by \eqref{eqn_118}}]
		\end{align*}
		and so, $\underline{V}$ is a $*$-regular dilation of $\underline{T}$. It follows from  Theorem \ref{thm_mainI} that $\underline{V}$ is doubly $q$-commuting.
		
		\vspace{0.2cm}
		
		\noindent $(2) \implies (1)$. Assume that $\underline{T}$ has a minimal regular $q$-isometric dilation, say $\underline{V}$, which is doubly $q$-commuting. It follows from Theorem \ref{thm_mainI} that $\underline{V}$ is a $*$-regular dilation of $\underline{T}$. Thus, we have 
		\begin{align*}
			& \quad \left[T_1^{*m_1^-}\dotsc T_k^{*m_k^-}\right]\left[T_1^{m_1^+} \dotsc T_k^{m_k^+}\right]\\
			&=P_\mathcal{H}\left[V_1^{*m_1^-}\dotsc V_k^{*m_k^-}\right]\left[V_1^{m_1^+} \dotsc V_k^{m_k^+}\right]\bigg|_\mathcal{H} \qquad [\text{as $\underline{V}$ is regular}]\\
			&=	\underset{1\leq i<j \leq k}{\prod}q_{ij}^{-m_i^-m_j^++m_i^+m_j^-} \left[T_1^{m_1^+} \dotsc T_k^{m_k^+}\right]\left[T_1^{*m_1^-}\dotsc T_k^{*m_k^-}\right] \qquad [\text{as $\underline{V}$ is $*$-regular}]
		\end{align*}
		for all $m_1, \dotsc, m_k \in \Z$ and so, $\underline{T}$ is doubly $q$-commuting. The proof is now complete.
	\end{proof}
	
	Recall that a coisometry is a Hilbert space operator whose adjoint is an isometry. A coisometry is said to be pure if its adjoint is a pure isometry. The next result is an follows directly from Theorem \ref{thm_mainI}, and we outline its proof briefly.
	
	\begin{cor}
		A $q$-commuting tuple of coisometries with $\|q\|=1$ has a minimal regular $q$-unitary dilation if and only if it is doubly $q$-commuting.
	\end{cor}
	
	\begin{proof}
		Let $\underline{T}=(T_1, \dotsc, T_k)$ be a $q$-commuting tuple of coisometries with $\|q\|=1$ acting on a Hilbert space $\HS$ and let $\underline{U}=(U_1, \dotsc, U_k)$ be a minimal regular $q$-unitary dilation of $\underline{T}$. By minimality, $\underline{U}$ acts on the space $\mathcal{K}=\ov{\text{span}}\left\{U_1^{m_1}\dotsc U_k^{m_k}x: m_1, \dotsc, m_k \in \Z, \ x \in \HS \right\}$. Moreover, $T_i^*=P_\HS U_i^*|_{\HS}$ for $1 \leq i \leq k$. Since $T_i^*$ and $U_i^*$ are isometries, $T_i^*=U_i^*|_{\HS}$ for $1 \leq i \leq k$. Then
		\[
		T_1^{m_1^+} \dotsc T_k^{m_k^+}T_1^{*m_1^-}\dotsc T_k^{*m_k^-}=T_1^{m_1^+} \dotsc T_k^{m_k^+}U_1^{*m_1^-}\dotsc U_k^{*m_k^-}\bigg|_\HS=P_\HS U_1^{m_1^+} \dotsc U_k^{m_k^+}U_1^{*m_1^-}\dotsc U_k^{*m_k^-}\bigg|_\HS
		\] 
		and so, $(U_1^*, \dotsc, U_k^*)$ is a  minimal regular $q$-unitary dilation of $(T_1^*, \dotsc, T_k^*)$. We have by Theorem \ref{thm_mainII} that $\underline{T}$ is doubly commuting. The converse follows directly from Theorem \ref{thm_mainII}.
	\end{proof}
	
Another application of Theorem \ref{thm_mainI} yields a structure theorem for $q$-commuting tuples of contractions with $\|q\|=1$ having a regular $q$-unitary dilation. To this end, consider a $q$-commuting tuple of contractions $\underline{T} = (T_1, \dotsc, T_k)$ with $\|q\|=1$ that admits a regular $q$-unitary dilation. Evidently, $\underline{T}$ has a minimal regular $q$-unitary dilation. By Corollary \ref{cor_112}, $(T_1^*, \dotsc, T_k^*)$ has a minimal $*$-regular $q$-isometric dilation, say $\underline{V}=(V_1, \dotsc, V_k)$, on a Hilbert space $\KS$ containing $\HS$. We have by Lemma \ref{lem_106} that $\HS$ is a joint invariant subspace for $V_1^*, \dotsc, V_k^*$ and $T_i=V_i^*|_\HS$ for $1 \leq i \leq k$. By Theorem \ref{thm_mainI}, $\underline{V}$ is doubly $q$-commuting. Thus, the tuple $\underline{N}$ given by 
	\[
	\underline{N}=(N_1, \dotsc, N_k)=(V_1^*, \dotsc, V_k^*)
	\]
	is doubly $q$-commuting and consists of coisometries such that $\underline{T}=(N_1|_\HS, \dotsc, N_k|_\HS)$. It follows from Wold type decomposition for doubly $q$-commuting tuple of isometries (see \cite{Maji, RI, RII}) that there exist $2^k$ joint reducing subspaces $\left\{\KS_\eta : \eta \subseteq \{1, \dotsc, k\}\right\}$ for $V_1, \dotsc, V_k$ such that 
	\[
	\KS=\bigoplus_{\eta \subseteq \{1, \dotsc, k\}} \KS_\eta \quad \text{and} \quad \underline{V}=\bigoplus_{\eta \subseteq \{1, \dotsc, k\}} \underline{V}_\eta,
	\]
where $\underline{V}_\eta=(V_1|_{\KS_\eta}, \dotsc, V_k|_{\KS_\eta})$. Moreover, for each  $\eta$, $V_i|_{\KS_\eta}$ is a unitary if $i \in \eta$, and pure isometry if $i \notin \eta$. Clearly, each $\underline{V}_\eta$ is doubly $q$-commuting. Putting everything together, $\underline{T}$ is the restriction to an invariant subspace of a doubly $q$-commuting tuple of coisometries given by  
\begin{align*}
\underline{N}=\bigoplus_{\eta \subseteq \{1, \dotsc, k\}} \underline{N}_\eta \in \mathcal{B}\left(\bigoplus_{\eta \subseteq \{1, \dotsc, k\}} \KS_\eta\right),
\end{align*}
where $\underline{N}_\eta=(N_1|_{\KS_\eta}, \dotsc, N_k|_{\KS_\eta})=(V_1^*|_{\KS_\eta}, \dotsc, V_k^*|_{\KS_\eta})$. So, for each $\eta \subseteq \{1, \dotsc, k\}$, $N_i|_{\KS_\eta}$ is a unitary if $i \in \eta$, and pure coisometry if $i \notin \eta$. Thus, we arrive at the following structure theorem.

	\begin{thm}
	Let $\underline{T}=(T_1, \dotsc, T_k)$ be a $q$-commuting tuple of contractions with $\|q\|=1$ having a regular $q$-unitary dilation. Then there exists a doubly $q$-commuting tuple of coisometries $\underline{N}=(N_1, \dotsc, N_k)$ such that $T_i=N_i|_{\HS}$ for $1 \leq i \leq k$. In particular, $\underline{T}$ is the restriction to an invariant subspace of an operator 
	\[
	\bigoplus_{\eta \subseteq \{1, \dotsc, k\}} \underline{N}_\eta \in \mathcal{B}\left(\bigoplus_{\eta \subseteq \{1, \dotsc, k\}} \KS_\eta\right), 
	\]
where $\underline{N}_\eta=(N_1|_{\KS_\eta}, \dotsc, N_k|_{\KS_\eta})$. Moreover, $N_i|_{\KS_\eta}$ is a unitary if $i \in \eta$, and pure coisometry if $i \notin \eta$ for each $\eta$. Also, the tuple $(N_1|_{\KS_\eta}, \dotsc, N_k|_{\KS_\eta})$ is doubly $q$-commuting for each $\eta \subseteq \{1, \dotsc, k\}$.
	\end{thm}

\end{document}